\documentclass[11pt]{article}

\usepackage[nottoc,notlot,notlof]{tocbibind}

\usepackage{amsmath,amssymb,amsfonts,amsthm}
\usepackage{latexsym}
\usepackage{tikz}
\usetikzlibrary{automata,positioning}
\usetikzlibrary{chains,fit,shapes}
\usetikzlibrary{calc}
\usetikzlibrary{arrows}
\usetikzlibrary{positioning,calc}
\usetikzlibrary{graphs}
\usetikzlibrary{graphs.standard}
\usetikzlibrary{arrows,decorations.markings}

\newtheorem{theorem}{Theorem}
\newtheorem{lemma}[theorem]{Lemma}

\newtheorem{problem}{Problem}

\newtheorem{conjecture}{Conjecture}

\title{On word-representability of simplified de Bruijn graphs}

\author{Anthony Petyuk\footnote{College of Letters and Science, University of California, Berkeley, California, USA. \\ {\bf Email:} anthony.petyuk@berkeley.edu}}

\begin{document}  

\maketitle

\abstract{A graph $G=(V,E)$ is word-representable if there exists a word $w$
over the alphabet $V$ such that letters $x$ and $y$ alternate in $w$ if and only if $xy\in E$. Word-representable graphs generalize several important classes of graphs such as $3$-colorable graphs, circle graphs, and comparability graphs. There is a long line of research in the literature dedicated to word-representable graphs. 

In this paper, we study word-representability of simplified de Bruijn graphs. The simplified de Bruijn graph $S(n,k)$ is a simple graph obtained from the de Bruijn graph $B(n,k)$ by removing orientations and loops and replacing multiple edges between a pair of vertices by a single edge. De Bruijn graphs are a key object in combinatorics on words that found numerous applications, in particular, in genome assembly. We show that binary simplified de Bruijn graphs (i.e.\ $S(n,2)$) are word-representable for any $n\geq 1$, while $S(2,k)$ and $S(3,k)$ are non-word-representable for $k\geq 3$. We conjecture that all simplified de Bruijn graphs $S(n,k)$ are non-word-rerpesentable for $n\geq 4$ and $k\geq 3$.}  

\section{Introduction}\label{sec1}

\subsection{Word-representable graphs and semi-transitive orientations.} Two letters $x$ and $y$ {\em alternate} in a word $w$ if after deleting in $w$ all letters but the copies of $x$ and $y$ we either obtain a word $xyxy\cdots$ (of even or odd length) or a word $yxyx\cdots$ (of even or odd length). A graph $G=(V,E)$ is word-representable if there exists a word $w$
over the alphabet $V$ such that letters $x$ and $y$ alternate in $w$ if and only if $xy\in E$. 

There is a long line of research papers dedicated to the theory of word-representable graphs (e.g.\ see \cite{K17,KL15} and references therein). The motivation to study these graphs is their relevance to algebra, graph theory, computer science, combinatorics on words, and scheduling \cite{KL15}. In particular, word-representable graphs generalize several important classes of graphs (e.g. circle graphs, $3$-colorable graphs and comparability graphs). By \cite{K17,KL15}, the minimal (by the number of vertices) non-word-representable graph is the wheel graph $W_5$ on six vertices (which is obtained by adding to the cycle graph $C_5$ an all-adjacent vertex; see Figure~\ref{W5}).

An orientation of a graph is {\it semi-transitive} if it is acyclic, and for any directed path $v_0\rightarrow v_1\rightarrow \cdots\rightarrow v_k$ either there is no arc between $v_0$ and $v_k$, or $v_i\rightarrow v_j$ is an arc for all $0\leq i<j\leq k$. An undirected graph is semi-transitive if it admits a semi-transitive orientation. A key result in the theory of word-representable graphs is the following theorem.

\begin{theorem}[\cite{HKP16}]\label{key-thm} A graph is word-representable if and only if it is semi-transitive. \end{theorem}

As an immediate corollary to Theorem~\ref{key-thm} we have the following result.

\begin{theorem}[\cite{HKP16}]\label{3-col} Any $3$-colorable graph is word-representable. \end{theorem}

Finally, we note the handy software to work with word-representable graphs~\cite{G,S}.

\subsection{Simplified de Bruijn graphs.}\label{subsec-simplified} Let $A=\{0,1,\ldots,k-1\}$ be a $k$-letter alphabet and $A^n$ be the set of all words over $A$ of length $n$. A {\em de Bruijn graph} $B(n,k)$ consists of $k^n$ vertices labeled by words in $A^n$ and its directed edges are $x_1x_2\cdots x_n\rightarrow x_2\cdots x_nx_{n+1}$ where $x_i\in A$ for $i\in\{1,2,\ldots,n+1\}$. De Bruijn graphs are an important structure in combinatorics on words that found wide applications, in particular, in genome assembly~\cite{CPG}.

The {\em simplified de Bruijn graph} $S(n,k)$ is a simple graph obtained from $B(n,k)$ by removing orientations and loops and replacing multiple edges between a pair of vertices by a single edge. To our best knowledge, the notion of a simplified de Bruijn graph is introduced in this paper for the first time. We label the edge between $x_1x_2\cdots x_n$ and $x_2x_3\cdots x_{n+1}$ in $S(n,k)$ by $x_1x_2\cdots x_{n+1}$.

\subsection{Results in this paper.} In Section~\ref{binary} we show that binary simplified de Bruijn graphs are word-representable. In Section~\ref{non-binary}  we show that $S(2,k)$ and $S(3,k)$ are non-word-representable for $k\geq 3$. In Section~\ref{concluding} we state a conjecture on full classification of word-representable simplified de Bruijn graphs along with an open problem.

\section{Word-representability of $S(n,2)$}\label{binary}

Let $x^k=\underbrace{x\ldots x}_{k\mbox{ \tiny{times}}}$ and $\bar{0}=1$ and $\bar{1}=0$. 

\begin{theorem} $S(n,2)$ is word-representable for any $n\geq 1$. \end{theorem}

\begin{proof} We will show that $S(n,2)$ is $3$-colorable so that the desired result will follow from Theorem~\ref{3-col}. Clearly, $S(1,2)$ is $3$-colorable, so that we can assume that $n\geq 2$. 

All vertices in $S(n,2)$ can be subdivided into the following three disjoint groups: vertices ending with $\bar{x}x^{2k}$ for $k\geq 1$ (where the $\bar{x}$ can be absent), those ending  with $10^{2k+1}$ for $k\geq 0$ (where the 1 can be absent), and those ending  with $01^{2k+1}$ for $k\geq 0$ (where the 0 can be absent). We colour the vertices in these groups Red, Blue and Green, respectively. 

Since $S(n,2)$ has no loops,  its edges correspond to the following two types of edges in $B(n,2)$:
\begin{itemize}
\item Edges $uv$ in $B(n,2)$,  where the words corresponding to $u$ and $v$ end in the same letter. Since one of $u$ or $v$ must end in an odd number of repeated letters and the other must end in an even number of repeated letters, these directed edges correspond to edges in $S(2,k)$ where the colours of $u$ and $v$ are always different.
  \item Edges $uv$ in $B(n,2)$,  where the words corresponding to $u$ and $v$ end in different letters. The vertex $v$ must end in one repeated letter $x$ (an odd number of letters), and $u$ must end in either an even or an odd number of repeated letters $\bar{x}$. These directed edges also correspond to edges in $S(2,k)$ where the colors of $u$ and $v$ are always different.
 \end{itemize}
 Thus, we have constructed a proper $3$-colouring of $S(2,k)$.
\end{proof}

\section{Non-word-representability of $S(2,k)$ and $S(3,k)$ for $k\geq 3$}\label{non-binary}

\begin{theorem} $S(2,k)$ is non-word-representable for any $k\geq 3$. \end{theorem}

\begin{proof} Since $S(2,k)$ is an induced subgraph of $S(2,k+1)$ for any $k\geq 1$, it is sufficient to prove that $S(2,3)$ is non-word-representable and to use the hereditary property of non-word-representability. We get non-word-representability of $S(2,3)$ by observing that it contains the non-word-representable graph $W_5$ as an induced subgraph given in Figure~\ref{W5}. 
%In that figure, we label edges as they would typically be labeled in $B(2,3)$, namely, an edge $xy\rightarrow yz$ is labeled by $xyz$.
\end{proof}

\begin{figure}
\begin{center}
\includegraphics[scale=0.4]{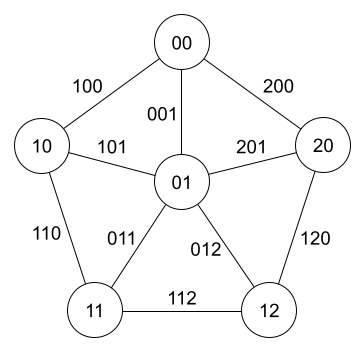}
\caption{An induced subgraph of $S(2,3)$ isomorphic to the wheel graph $W_5$. The node labels correspond to words in $A^2$ and edges are labeled according to the description in Section~\ref{subsec-simplified}.}\label{W5}
\end{center}
\end{figure}

\begin{theorem}\label{S3k-thm} $S(3,k)$ is non-word-representable for any $k\geq 3$.  \end{theorem}

\begin{proof} Since $S(3,k)$ is an induced subgraph of $S(3,k+1)$ for any $k\geq 1$, it is sufficient to prove that $S(3,3)$ is non-word-representable. Figure~\ref{S33-subgraph} presents an induced subgraph $S$ of $S(3,3)$, where we label edges according to the description in Section~\ref{subsec-simplified} and label each vertex using two labels, one being the label of the vertex coming from $B(3,3)$, and the other one is a technical label that we need to present our proof; for example, the vertices 102 and 210 are also labeled by 1 and 2, respectively.

\begin{figure}[ht]
\begin{center}
\includegraphics[scale=0.4]{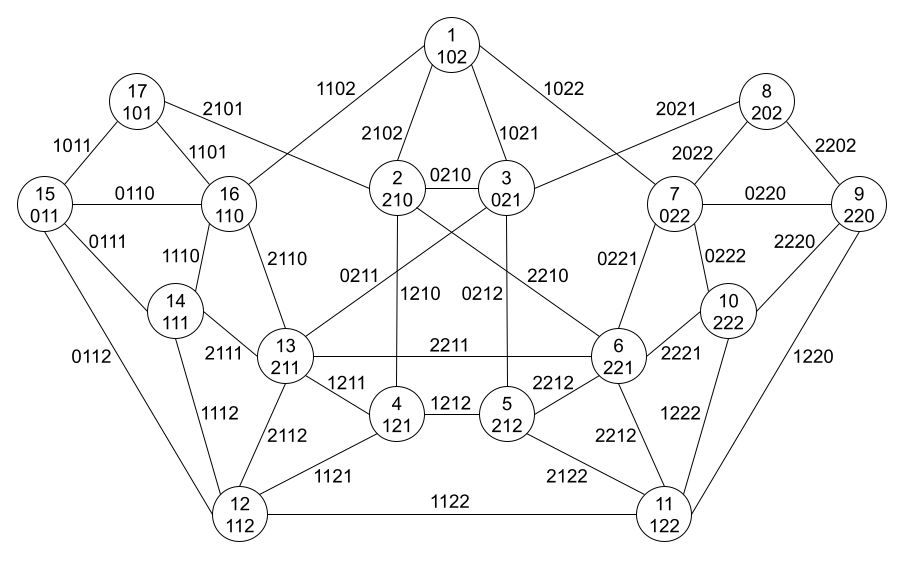}
\caption{A non-word-representable induced subgraph $S$ of $S(3,3)$. The lower node labels and edge labels are as described in Section~\ref{subsec-simplified}. The upper node labels correspond to those used in the computer-generated proof presented in the appendix.}\label{S33-subgraph}
\end{center}
\end{figure}

Non-word-representability of $S$ can be checked using two independent user-friendly pieces of software \cite{G,S}. Moreover, \cite{S} can produce a 100 ``line'' formal proof of non-word-representability of $S$ that, in principle, can be verified by a human, but it is too long to be included here, so it is presented in the appendix along with instructions on how to understand the proof (coming from \cite{KS}).  
\end{proof}

\section{Concluding remarks}\label{concluding}

In this paper, we show that binary simplified de Bruijn graphs are word-representable, while $S(2,k)$ and $S(3,k)$, for $k\geq 3$, are non-word-representable. We conclude our paper with the following conjecture and an open problem.

\begin{conjecture} $S(n,k)$ is non-word-rerpesentable for $n\geq 4$ and $k\geq 3$. \end{conjecture}

\begin{problem} Is it true that the graph $S$ in Figure~\ref{S33-subgraph} is a minimal (by the number of vertices) non-word-representable induced subgraph of the graph $S(3,3)$? Finding a smaller non-word-representable induced subgraph  might allow, using the software \cite{S}, shortening the formal proof of non-word-representability of  $S(3,3)$ in the proof of Theorem~\ref{S3k-thm}.\end{problem}

\appendix
\section*{Appendix: Proof of non-word-representability of the graph $S$ in Theorem~\ref{S3k-thm}}\label{appendix}

Our proof uses the following lemmas.

\begin{lemma}[\cite{KP}]\label{lemma} Suppose that an undirected graph $G$ has a cycle $C=x_1x_2\cdots x_mx_1$, where $m\geq 4$ and the vertices in $\{x_1,x_2,\ldots,x_m\}$ do not induce a clique in $G$.  If $G$ is oriented semi-transitively, and $m-2$ edges of $C$ are oriented in the same direction (i.e. from $x_i$ to $x_{i+1}$ or vice versa, where the index $m+1:=1$) then the remaining two edges of $C$ are oriented in the opposite direction.\end{lemma}

\begin{lemma}[\cite{KS}]\label{appendix2}
If $G$ is word-representable and $u$ is an arbitrary vertex in $G$, then there exists a semi-transitive orientation of $G$ with source~$u$.
\end{lemma} 

We refer to \cite{KS} for more details on the format of the proof below. By a ``line'' of a proof we mean a sequence of instructions that directs us in orienting a partially oriented graph and necessarily ends with detecting a shortcut or another contradiction showing that this particular orientation branch will not produce a semi-transitive orientation. If no branch produces a semi-transitive orientation then the graph is non-semi-transitively orientable (and hence non-word-representable by Theorem~\ref{key-thm}). 

The proof begins with $15\rightarrow 17$ showing the orientation of the edge $(15)(17)$, which is done W.L.O.G.\ because reversing all orientations in a semi-transitively oriented graph produces a semi-transitively oriented graph. Further, there are four types of instructions:
\begin{itemize}
\item ``MC'' followed by a number $X$ means ``Move to Copy $X$'', where Copy $X$ of the graph in question is a partially oriented version of the graph that has been created at some point in the branching process. This instruction is always followed by an oriented edge $A\rightarrow B$ reminding on the directed edge obtained after application of the branching process; see description of ``B'' to be discussed next. 
\item ``B'' followed by ``$X\rightarrow Y$ (Copy $Z$)'' means ``Branch on edge $XY$, orient the edge as $X\rightarrow Y$, create a copy of the current version of the graph except orient the edge $XY$ there as $Y\rightarrow X$, and call the new copy $Z$; leave Copy $Z$ aside and continue to follow the instructions''. The instruction B occurs when the software detects that no application of Lemma~\ref{lemma} is possible in the partially oriented graph.
\item One ``O'' followed by ``$X\rightarrow Y$'', in turn followed by, in brackets, ``C'' followed by a cycle ``$X$-$Y$-$Z$''. This instruction tells us to orient the edge $XY$ as $X\rightarrow Y$ because otherwise, in the triangle $XYZ$, we would get a directed cycle. If instead of a triangle we see a longer cycle, then we deal with an application of Lemma~\ref{lemma} to a cycle where all but two edges are oriented in one direction, and one of the remaining two edges is oriented in the opposite direction.
\item Two ``O''s followed by ``$X\rightarrow Y$'', in turn followed by, in brackets, ``C'' followed by a cycle ``$X$-$Y$-$Z$-$\cdots$''. This instruction tells us to which cycle Lemma~\ref{lemma} can be applied and which edges will become oriented.
\end{itemize}

\noindent
Each line ends with ``$S:X-Y-\cdots-Z$'' indicating that a shortcut with the shortcutting edge $X\rightarrow Z$ is obtained. \\

\noindent
{\bf The proof:} 

\noindent
By Lemma~\ref{appendix2}, we can assume that the vertex 13 (of the maximum degree) is a source. The rest of the proof goes as follows. \\

\begin{tiny}
\noindent
\noindent
{\bf 1.} B14$\rightarrow$16 (Copy 2) O14$\rightarrow$12 (C12-14-16-13) O4$\rightarrow$12 (C4-13-14-12) B5$\rightarrow$6 (Copy 3) O5$\rightarrow$3 (C3-13-6-5) O5$\rightarrow$4 (C3-13-4-5) O11$\rightarrow$12 O5$\rightarrow$11 (C4-12-11-5) O11$\rightarrow$6 (C6-13-12-11) B2$\rightarrow$6 (Copy 4) O2$\rightarrow$3 (C2-6-5-3) O2$\rightarrow$4 (C2-4-5-3) B6$\rightarrow$10 (Copy 5) O11$\rightarrow$10 (C6-11-10) O11$\rightarrow$9 O9$\rightarrow$10 (C6-11-9-10) B8$\rightarrow$9 (Copy 6) O7$\rightarrow$10 O8$\rightarrow$7 (C7-10-9-8) O7$\rightarrow$6 (C6-11-10-7) O7$\rightarrow$9 (C6-10-9-7) O8$\rightarrow$3 (C2-6-7-8-3) B12$\rightarrow$15 (Copy 7) O14$\rightarrow$15 (C12-15-14) O16$\rightarrow$15 (C12-15-16-13) O2$\rightarrow$17 O17$\rightarrow$15 (C2-17-15-12-4) O17$\rightarrow$16 (C14-16-17-15) O1$\rightarrow$16 O2$\rightarrow$1 (C1-16-17-2) O7$\rightarrow$1 (C1-7-6-2) O3$\rightarrow$1 (C1-7-8-3) S:13-3-1-16

\noindent
{\bf 2.} MC7 15$\rightarrow$12 O15$\rightarrow$14 (C12-15-14-13) O15$\rightarrow$16 (C14-16-15) O17$\rightarrow$16 O15$\rightarrow$17 (C14-16-17-15) O2$\rightarrow$17 (C2-17-15-12-4) O1$\rightarrow$16 O2$\rightarrow$1 (C1-16-17-2) O7$\rightarrow$1 (C1-7-6-2) O3$\rightarrow$1 (C1-7-8-3) S:13-3-1-16

\noindent
{\bf 3.} MC6 9$\rightarrow$8 O3$\rightarrow$8 (C3-8-9-11-5) B15$\rightarrow$16 (Copy 8) O15$\rightarrow$12 (C12-15-16-13) O15$\rightarrow$14 (C12-15-14-13) O17$\rightarrow$16 O15$\rightarrow$17 (C14-16-17-15) O2$\rightarrow$17 (C2-17-15-12-4) O1$\rightarrow$16 O2$\rightarrow$1 (C1-16-17-2) O1$\rightarrow$3 (C1-16-13-3) O1$\rightarrow$7 O7$\rightarrow$8 (C1-7-8-3) O6$\rightarrow$7 (C1-7-6-2) O9$\rightarrow$7 (C6-10-9-7) O10$\rightarrow$7 (C6-11-10-7) S:9-10-7-8

\noindent
{\bf 4.} MC8 16$\rightarrow$15 O14$\rightarrow$15 (C14-16-15) O12$\rightarrow$15 (C12-15-14-13) O2$\rightarrow$17 O17$\rightarrow$15 (C2-17-15-12-4) O17$\rightarrow$16 (C14-16-17-15) O1$\rightarrow$16 O2$\rightarrow$1 (C1-16-17-2) O1$\rightarrow$3 (C1-16-13-3) O1$\rightarrow$7 O7$\rightarrow$8 (C1-7-8-3) O6$\rightarrow$7 (C1-7-6-2) O9$\rightarrow$7 (C6-10-9-7) O10$\rightarrow$7 (C6-11-10-7) S:9-10-7-8

\noindent
{\bf 5.} MC5 10$\rightarrow$6 O10$\rightarrow$11 (C5-11-10-6) O10$\rightarrow$7 O7$\rightarrow$6 (C6-11-10-7) B7$\rightarrow$8 (Copy 9) O10$\rightarrow$9 O9$\rightarrow$8 (C7-10-9-8) O7$\rightarrow$9 (C6-10-9-7) O11$\rightarrow$9 (C6-11-9-7) O3$\rightarrow$8 (C3-8-9-11-5) B12$\rightarrow$15 (Copy 10) O14$\rightarrow$15 (C12-15-14) O16$\rightarrow$15 (C12-15-16-13) O2$\rightarrow$17 O17$\rightarrow$15 (C2-17-15-12-4) O17$\rightarrow$16 (C14-16-17-15) O1$\rightarrow$16 O2$\rightarrow$1 (C1-16-17-2) O7$\rightarrow$1 (C1-7-6-2) O3$\rightarrow$1 (C1-7-8-3) S:13-3-1-16

\noindent
{\bf 6.} MC10 15$\rightarrow$12 O15$\rightarrow$14 (C12-15-14-13) O15$\rightarrow$16 (C14-16-15) O17$\rightarrow$16 O15$\rightarrow$17 (C14-16-17-15) O2$\rightarrow$17 (C2-17-15-12-4) O1$\rightarrow$16 O2$\rightarrow$1 (C1-16-17-2) O7$\rightarrow$1 (C1-7-6-2) O3$\rightarrow$1 (C1-7-8-3) S:13-3-1-16

\noindent
{\bf 7.} MC9 8$\rightarrow$7 O8$\rightarrow$3 (C2-6-7-8-3) B15$\rightarrow$16 (Copy 11) O15$\rightarrow$12 (C12-15-16-13) O15$\rightarrow$14 (C12-15-14-13) O17$\rightarrow$16 O15$\rightarrow$17 (C14-16-17-15) O2$\rightarrow$17 (C2-17-15-12-4) O1$\rightarrow$16 O2$\rightarrow$1 (C1-16-17-2) O7$\rightarrow$1 (C1-7-6-2) O3$\rightarrow$1 (C1-7-8-3) S:13-3-1-16

\noindent
{\bf 8.} MC11 16$\rightarrow$15 O14$\rightarrow$15 (C14-16-15) O12$\rightarrow$15 (C12-15-14-13) O2$\rightarrow$17 O17$\rightarrow$15 (C2-17-15-12-4) O17$\rightarrow$16 (C14-16-17-15) O1$\rightarrow$16 O2$\rightarrow$1 (C1-16-17-2) O7$\rightarrow$1 (C1-7-6-2) O3$\rightarrow$1 (C1-7-8-3) S:13-3-1-16

\noindent
{\bf 9.} MC4 6$\rightarrow$2 O3$\rightarrow$2 (C2-6-5-3) O4$\rightarrow$2 (C2-4-5-3) B6$\rightarrow$10 (Copy 12) O11$\rightarrow$10 (C6-11-10) O11$\rightarrow$9 O9$\rightarrow$10 (C6-11-9-10) B8$\rightarrow$9 (Copy 13) O7$\rightarrow$10 O8$\rightarrow$7 (C7-10-9-8) O7$\rightarrow$6 (C6-11-10-7) O7$\rightarrow$1 O1$\rightarrow$2 (C1-7-6-2) O8$\rightarrow$3 O3$\rightarrow$1 (C1-7-8-3) O16$\rightarrow$1 (C1-16-13-3) O16$\rightarrow$17 O17$\rightarrow$2 (C1-16-17-2) O7$\rightarrow$9 (C6-10-9-7) O15$\rightarrow$17 O14$\rightarrow$15 (C14-16-17-15) O12$\rightarrow$15 (C12-15-14-13) S:4-12-15-17-2

\noindent
{\bf 10.} MC13 9$\rightarrow$8 O3$\rightarrow$8 (C3-8-9-11-5) B15$\rightarrow$16 (Copy 14) O15$\rightarrow$12 (C12-15-16-13) O15$\rightarrow$14 (C12-15-14-13) O17$\rightarrow$16 O15$\rightarrow$17 (C14-16-17-15) O17$\rightarrow$2 (C2-17-16-13-3) B7$\rightarrow$10 (Copy 15) O7$\rightarrow$6 (C6-11-10-7) O7$\rightarrow$1 O1$\rightarrow$2 (C1-7-6-2) O1$\rightarrow$16 (C1-16-17-2) O7$\rightarrow$8 O3$\rightarrow$1 (C1-7-8-3) S:13-3-1-16

\noindent
{\bf 11.} MC15 10$\rightarrow$7 O6$\rightarrow$7 (C6-10-7) O9$\rightarrow$7 (C7-10-9) O8$\rightarrow$7 (C7-10-9-8) O1$\rightarrow$7 O3$\rightarrow$1 (C1-7-8-3) O1$\rightarrow$2 (C1-7-6-2) O1$\rightarrow$16 (C1-16-17-2) S:13-3-1-16

\noindent
{\bf 12.} MC14 16$\rightarrow$15 O14$\rightarrow$15 (C14-16-15) O12$\rightarrow$15 (C12-15-14-13) B2$\rightarrow$17 (Copy 16) O16$\rightarrow$17 (C2-17-16-13-3) O15$\rightarrow$17 (C14-16-17-15) B7$\rightarrow$10 (Copy 17) O7$\rightarrow$6 (C6-11-10-7) O7$\rightarrow$1 O1$\rightarrow$2 (C1-7-6-2) O1$\rightarrow$16 (C1-16-17-2) O7$\rightarrow$8 O3$\rightarrow$1 (C1-7-8-3) S:13-3-1-16

\noindent
{\bf 13.} MC17 10$\rightarrow$7 O6$\rightarrow$7 (C6-10-7) O9$\rightarrow$7 (C7-10-9) O8$\rightarrow$7 (C7-10-9-8) O1$\rightarrow$7 O3$\rightarrow$1 (C1-7-8-3) O1$\rightarrow$2 (C1-7-6-2) O1$\rightarrow$16 (C1-16-17-2) S:13-3-1-16

\noindent
{\bf 14.} MC16 17$\rightarrow$2 O17$\rightarrow$15 (C2-17-15-12-4) O17$\rightarrow$16 (C14-16-17-15) B7$\rightarrow$10 (Copy 18) O7$\rightarrow$6 (C6-11-10-7) O7$\rightarrow$1 O1$\rightarrow$2 (C1-7-6-2) O1$\rightarrow$16 (C1-16-17-2) O7$\rightarrow$8 O3$\rightarrow$1 (C1-7-8-3) S:13-3-1-16

\noindent
{\bf 15.} MC18 10$\rightarrow$7 O6$\rightarrow$7 (C6-10-7) O9$\rightarrow$7 (C7-10-9) O8$\rightarrow$7 (C7-10-9-8) O1$\rightarrow$7 O3$\rightarrow$1 (C1-7-8-3) O1$\rightarrow$2 (C1-7-6-2) O1$\rightarrow$16 (C1-16-17-2) S:13-3-1-16

\noindent
{\bf 16.} MC12 10$\rightarrow$6 O10$\rightarrow$11 (C5-11-10-6) O10$\rightarrow$7 O7$\rightarrow$6 (C6-11-10-7) O7$\rightarrow$1 O1$\rightarrow$2 (C1-7-6-2) B1$\rightarrow$16 (Copy 19) O1$\rightarrow$3 (C1-16-13-3) O7$\rightarrow$8 O8$\rightarrow$3 (C1-7-8-3) O10$\rightarrow$9 O9$\rightarrow$8 (C7-10-9-8) O7$\rightarrow$9 (C6-10-9-7) O11$\rightarrow$9 (C6-11-9-7) S:5-11-9-8-3

\noindent
{\bf 17.} MC19 16$\rightarrow$1 O16$\rightarrow$17 O17$\rightarrow$2 (C1-16-17-2) O3$\rightarrow$1 (C1-16-13-3) O15$\rightarrow$17 O14$\rightarrow$15 (C14-16-17-15) O12$\rightarrow$15 (C12-15-14-13) S:4-12-15-17-2

\noindent
{\bf 18.} MC3 6$\rightarrow$5 O3$\rightarrow$5 (C3-13-6-5) O4$\rightarrow$5 (C3-13-4-5) B2$\rightarrow$6 (Copy 20) O2$\rightarrow$3 (C2-6-5-3) O2$\rightarrow$4 (C2-4-5-3) B2$\rightarrow$17 (Copy 21) B16$\rightarrow$17 (Copy 22) O15$\rightarrow$17 O14$\rightarrow$15 (C14-16-17-15) O12$\rightarrow$15 (C12-15-14-13) S:2-4-12-15-17

\noindent
{\bf 19.} MC22 17$\rightarrow$16 O1$\rightarrow$16 O2$\rightarrow$1 (C1-16-17-2) O1$\rightarrow$3 (C1-16-13-3) B6$\rightarrow$11 (Copy 23) O12$\rightarrow$11 (C6-13-12-11) O5$\rightarrow$11 (C4-12-11-5) O10$\rightarrow$11 O6$\rightarrow$10 (C5-11-10-6) B9$\rightarrow$10 (Copy 24) O9$\rightarrow$11 (C9-11-10) B1$\rightarrow$7 (Copy 25) O6$\rightarrow$7 (C1-7-6-2) O9$\rightarrow$7 (C6-10-9-7) O10$\rightarrow$7 (C6-11-10-7) O9$\rightarrow$8 O8$\rightarrow$7 (C7-10-9-8) O8$\rightarrow$3 (C1-7-8-3) S:9-8-3-5-11

\noindent
{\bf 20.} MC25 7$\rightarrow$1 O7$\rightarrow$6 (C1-7-6-2) O7$\rightarrow$10 (C6-10-7) O7$\rightarrow$8 O8$\rightarrow$3 (C1-7-8-3) O7$\rightarrow$9 (C6-10-9-7) O9$\rightarrow$8 (C7-10-9-8) S:9-8-3-5-11

\noindent
{\bf 21.} MC24 10$\rightarrow$9 O7$\rightarrow$9 O6$\rightarrow$7 (C6-10-9-7) O1$\rightarrow$7 (C1-7-6-2) O11$\rightarrow$9 (C6-11-9-7) O10$\rightarrow$7 (C6-11-10-7) O3$\rightarrow$8 O8$\rightarrow$9 (C3-8-9-11-5) O7$\rightarrow$8 (C1-7-8-3) S:10-7-8-9

\noindent
{\bf 22.} MC23 11$\rightarrow$6 O11$\rightarrow$5 (C5-11-6) O11$\rightarrow$12 (C4-12-11-5) B1$\rightarrow$7 (Copy 26) O6$\rightarrow$7 (C1-7-6-2) O11$\rightarrow$9 O9$\rightarrow$7 (C6-11-9-7) O11$\rightarrow$10 O10$\rightarrow$7 (C6-11-10-7) O6$\rightarrow$10 (C5-11-10-6) O9$\rightarrow$10 (C6-10-9-7) O9$\rightarrow$8 O8$\rightarrow$7 (C7-10-9-8) O8$\rightarrow$3 (C1-7-8-3) S:11-9-8-3-5

\noindent
{\bf 23.} MC26 7$\rightarrow$1 O7$\rightarrow$6 (C1-7-6-2) O7$\rightarrow$8 O8$\rightarrow$3 (C1-7-8-3) B7$\rightarrow$10 (Copy 27) O11$\rightarrow$10 (C6-11-10-7) O6$\rightarrow$10 (C5-11-10-6) O9$\rightarrow$10 O7$\rightarrow$9 (C6-10-9-7) O11$\rightarrow$9 (C6-11-9-7) O9$\rightarrow$8 (C7-10-9-8) S:11-9-8-3-5

\noindent
{\bf 24.} MC27 10$\rightarrow$7 O10$\rightarrow$6 (C6-10-7) O10$\rightarrow$11 (C5-11-10-6) O10$\rightarrow$9 O9$\rightarrow$8 (C7-10-9-8) O7$\rightarrow$9 (C6-10-9-7) O11$\rightarrow$9 (C6-11-9-7) S:11-9-8-3-5

\noindent
{\bf 25.} MC21 17$\rightarrow$2 O17$\rightarrow$16 (C2-17-16-13-3) O17$\rightarrow$15 O15$\rightarrow$12 (C2-17-15-12-4) O15$\rightarrow$14 (C12-15-14-13) S:17-15-14-16

\noindent
{\bf 26.} MC20 6$\rightarrow$2 O3$\rightarrow$2 (C2-6-5-3) O4$\rightarrow$2 (C2-4-5-3) B2$\rightarrow$17 (Copy 28) O16$\rightarrow$17 (C2-17-16-13-3) O15$\rightarrow$17 O14$\rightarrow$15 (C14-16-17-15) O12$\rightarrow$15 (C12-15-14-13) O16$\rightarrow$15 (C12-15-16-13) B6$\rightarrow$11 (Copy 29) O12$\rightarrow$11 (C6-13-12-11) O5$\rightarrow$11 (C4-12-11-5) O10$\rightarrow$11 O6$\rightarrow$10 (C5-11-10-6) B6$\rightarrow$7 (Copy 30) O10$\rightarrow$7 (C6-11-10-7) B7$\rightarrow$8 (Copy 31) O10$\rightarrow$9 O9$\rightarrow$8 (C7-10-9-8) O7$\rightarrow$9 (C6-10-9-7) O11$\rightarrow$9 (C6-11-9-7) O3$\rightarrow$8 (C2-6-7-8-3) S:3-5-11-9-8

\noindent
{\bf 27.} MC31 8$\rightarrow$7 B3$\rightarrow$8 (Copy 32) O1$\rightarrow$7 O3$\rightarrow$1 (C1-7-8-3) O1$\rightarrow$2 (C1-7-6-2) O1$\rightarrow$16 (C1-16-17-2) S:13-3-1-16

\noindent
{\bf 28.} MC32 8$\rightarrow$3 O8$\rightarrow$9 O9$\rightarrow$11 (C3-8-9-11-5) O9$\rightarrow$7 (C6-11-9-7) O9$\rightarrow$10 (C6-10-9-7) S:8-9-10-7

\noindent
{\bf 29.} MC30 7$\rightarrow$6 O7$\rightarrow$10 (C6-10-7) O7$\rightarrow$1 O1$\rightarrow$2 (C1-7-6-2) O1$\rightarrow$16 (C1-16-17-2) O1$\rightarrow$3 (C1-16-13-3) O7$\rightarrow$8 O8$\rightarrow$3 (C1-7-8-3) O9$\rightarrow$10 O7$\rightarrow$9 (C6-10-9-7) O9$\rightarrow$11 (C9-11-10) O9$\rightarrow$8 (C7-10-9-8) S:9-8-3-5-11

\noindent
{\bf 30.} MC29 11$\rightarrow$6 O11$\rightarrow$5 (C5-11-6) O11$\rightarrow$12 (C4-12-11-5) B6$\rightarrow$10 (Copy 33) O11$\rightarrow$10 (C6-11-10) O11$\rightarrow$9 O9$\rightarrow$10 (C6-11-9-10) B8$\rightarrow$9 (Copy 34) O7$\rightarrow$10 O8$\rightarrow$7 (C7-10-9-8) O7$\rightarrow$6 (C6-11-10-7) O7$\rightarrow$1 O1$\rightarrow$2 (C1-7-6-2) O1$\rightarrow$16 (C1-16-17-2) O8$\rightarrow$3 O3$\rightarrow$1 (C1-7-8-3) S:13-3-1-16

\noindent
{\bf 31.} MC34 9$\rightarrow$8 O3$\rightarrow$8 (C3-8-9-11-5) B7$\rightarrow$10 (Copy 35) O7$\rightarrow$6 (C6-11-10-7) O7$\rightarrow$1 O1$\rightarrow$2 (C1-7-6-2) O1$\rightarrow$16 (C1-16-17-2) O7$\rightarrow$8 O3$\rightarrow$1 (C1-7-8-3) S:13-3-1-16

\noindent
{\bf 32.} MC35 10$\rightarrow$7 O6$\rightarrow$7 (C6-10-7) O9$\rightarrow$7 (C7-10-9) O8$\rightarrow$7 (C7-10-9-8) O1$\rightarrow$7 O3$\rightarrow$1 (C1-7-8-3) O1$\rightarrow$2 (C1-7-6-2) O1$\rightarrow$16 (C1-16-17-2) S:13-3-1-16

\noindent
{\bf 33.} MC33 10$\rightarrow$6 O10$\rightarrow$11 (C5-11-10-6) O10$\rightarrow$7 O7$\rightarrow$6 (C6-11-10-7) O7$\rightarrow$1 O1$\rightarrow$2 (C1-7-6-2) O1$\rightarrow$16 (C1-16-17-2) O1$\rightarrow$3 (C1-16-13-3) O7$\rightarrow$8 O8$\rightarrow$3 (C1-7-8-3) O10$\rightarrow$9 O9$\rightarrow$8 (C7-10-9-8) O7$\rightarrow$9 (C6-10-9-7) O11$\rightarrow$9 (C6-11-9-7) S:11-9-8-3-5

\noindent
{\bf 34.} MC28 17$\rightarrow$2 B16$\rightarrow$17 (Copy 36) O16$\rightarrow$1 O1$\rightarrow$2 (C1-16-17-2) O3$\rightarrow$1 (C1-16-13-3) O15$\rightarrow$17 O14$\rightarrow$15 (C14-16-17-15) O12$\rightarrow$15 (C12-15-14-13) S:4-12-15-17-2

\noindent
{\bf 35.} MC36 17$\rightarrow$16 B15$\rightarrow$17 (Copy 37) O15$\rightarrow$16 (C15-17-16) O15$\rightarrow$12 (C12-15-16-13) O15$\rightarrow$14 (C12-15-14-13) B6$\rightarrow$11 (Copy 38) O12$\rightarrow$11 (C6-13-12-11) O5$\rightarrow$11 (C4-12-11-5) O10$\rightarrow$11 O6$\rightarrow$10 (C5-11-10-6) B6$\rightarrow$7 (Copy 39) O10$\rightarrow$7 (C6-11-10-7) B7$\rightarrow$8 (Copy 40) O10$\rightarrow$9 O9$\rightarrow$8 (C7-10-9-8) O7$\rightarrow$9 (C6-10-9-7) O11$\rightarrow$9 (C6-11-9-7) O3$\rightarrow$8 (C2-6-7-8-3) S:3-5-11-9-8

\noindent
{\bf 36.} MC40 8$\rightarrow$7 B3$\rightarrow$8 (Copy 41) O1$\rightarrow$7 O3$\rightarrow$1 (C1-7-8-3) O1$\rightarrow$2 (C1-7-6-2) O1$\rightarrow$16 (C1-16-17-2) S:13-3-1-16

\noindent
{\bf 37.} MC41 8$\rightarrow$3 O8$\rightarrow$9 O9$\rightarrow$11 (C3-8-9-11-5) O9$\rightarrow$7 (C6-11-9-7) O9$\rightarrow$10 (C6-10-9-7) S:8-9-10-7

\noindent
{\bf 38.} MC39 7$\rightarrow$6 O7$\rightarrow$10 (C6-10-7) O7$\rightarrow$1 O1$\rightarrow$2 (C1-7-6-2) O1$\rightarrow$16 (C1-16-17-2) O1$\rightarrow$3 (C1-16-13-3) O7$\rightarrow$8 O8$\rightarrow$3 (C1-7-8-3) O9$\rightarrow$10 O7$\rightarrow$9 (C6-10-9-7) O9$\rightarrow$11 (C9-11-10) O9$\rightarrow$8 (C7-10-9-8) S:9-8-3-5-11

\noindent
{\bf 39.} MC38 11$\rightarrow$6 O11$\rightarrow$5 (C5-11-6) O11$\rightarrow$12 (C4-12-11-5) B6$\rightarrow$10 (Copy 42) O11$\rightarrow$10 (C6-11-10) O11$\rightarrow$9 O9$\rightarrow$10 (C6-11-9-10) B8$\rightarrow$9 (Copy 43) O7$\rightarrow$10 O8$\rightarrow$7 (C7-10-9-8) O7$\rightarrow$6 (C6-11-10-7) O7$\rightarrow$1 O1$\rightarrow$2 (C1-7-6-2) O1$\rightarrow$16 (C1-16-17-2) O8$\rightarrow$3 O3$\rightarrow$1 (C1-7-8-3) S:13-3-1-16

\noindent
{\bf 40.} MC43 9$\rightarrow$8 O3$\rightarrow$8 (C3-8-9-11-5) B7$\rightarrow$10 (Copy 44) O7$\rightarrow$6 (C6-11-10-7) O7$\rightarrow$1 O1$\rightarrow$2 (C1-7-6-2) O1$\rightarrow$16 (C1-16-17-2) O7$\rightarrow$8 O3$\rightarrow$1 (C1-7-8-3) S:13-3-1-16

\noindent
{\bf 41.} MC44 10$\rightarrow$7 O6$\rightarrow$7 (C6-10-7) O9$\rightarrow$7 (C7-10-9) O8$\rightarrow$7 (C7-10-9-8) O1$\rightarrow$7 O3$\rightarrow$1 (C1-7-8-3) O1$\rightarrow$2 (C1-7-6-2) O1$\rightarrow$16 (C1-16-17-2) S:13-3-1-16

\noindent
{\bf 42.} MC42 10$\rightarrow$6 O10$\rightarrow$11 (C5-11-10-6) O10$\rightarrow$7 O7$\rightarrow$6 (C6-11-10-7) O7$\rightarrow$1 O1$\rightarrow$2 (C1-7-6-2) O1$\rightarrow$16 (C1-16-17-2) O1$\rightarrow$3 (C1-16-13-3) O7$\rightarrow$8 O8$\rightarrow$3 (C1-7-8-3) O10$\rightarrow$9 O9$\rightarrow$8 (C7-10-9-8) O7$\rightarrow$9 (C6-10-9-7) O11$\rightarrow$9 (C6-11-9-7) S:11-9-8-3-5

\noindent
{\bf 43.} MC37 17$\rightarrow$15 O14$\rightarrow$15 (C14-16-17-15) O12$\rightarrow$15 (C12-15-14-13) O16$\rightarrow$15 (C12-15-16-13) B6$\rightarrow$11 (Copy 45) O12$\rightarrow$11 (C6-13-12-11) O5$\rightarrow$11 (C4-12-11-5) O10$\rightarrow$11 O6$\rightarrow$10 (C5-11-10-6) B6$\rightarrow$7 (Copy 46) O10$\rightarrow$7 (C6-11-10-7) B7$\rightarrow$8 (Copy 47) O10$\rightarrow$9 O9$\rightarrow$8 (C7-10-9-8) O7$\rightarrow$9 (C6-10-9-7) O11$\rightarrow$9 (C6-11-9-7) O3$\rightarrow$8 (C2-6-7-8-3) S:3-5-11-9-8

\noindent
{\bf 44.} MC47 8$\rightarrow$7 B3$\rightarrow$8 (Copy 48) O1$\rightarrow$7 O3$\rightarrow$1 (C1-7-8-3) O1$\rightarrow$2 (C1-7-6-2) O1$\rightarrow$16 (C1-16-17-2) S:13-3-1-16

\noindent
{\bf 45.} MC48 8$\rightarrow$3 O8$\rightarrow$9 O9$\rightarrow$11 (C3-8-9-11-5) O9$\rightarrow$7 (C6-11-9-7) O9$\rightarrow$10 (C6-10-9-7) S:8-9-10-7

\noindent
{\bf 46.} MC46 7$\rightarrow$6 O7$\rightarrow$10 (C6-10-7) O7$\rightarrow$1 O1$\rightarrow$2 (C1-7-6-2) O1$\rightarrow$16 (C1-16-17-2) O1$\rightarrow$3 (C1-16-13-3) O7$\rightarrow$8 O8$\rightarrow$3 (C1-7-8-3) O9$\rightarrow$10 O7$\rightarrow$9 (C6-10-9-7) O9$\rightarrow$11 (C9-11-10) O9$\rightarrow$8 (C7-10-9-8) S:9-8-3-5-11

\noindent
{\bf 47.} MC45 11$\rightarrow$6 O11$\rightarrow$5 (C5-11-6) O11$\rightarrow$12 (C4-12-11-5) B6$\rightarrow$10 (Copy 49) O11$\rightarrow$10 (C6-11-10) O11$\rightarrow$9 O9$\rightarrow$10 (C6-11-9-10) B8$\rightarrow$9 (Copy 50) O7$\rightarrow$10 O8$\rightarrow$7 (C7-10-9-8) O7$\rightarrow$6 (C6-11-10-7) O7$\rightarrow$1 O1$\rightarrow$2 (C1-7-6-2) O1$\rightarrow$16 (C1-16-17-2) O8$\rightarrow$3 O3$\rightarrow$1 (C1-7-8-3) S:13-3-1-16

\noindent
{\bf 48.} MC50 9$\rightarrow$8 O3$\rightarrow$8 (C3-8-9-11-5) B7$\rightarrow$10 (Copy 51) O7$\rightarrow$6 (C6-11-10-7) O7$\rightarrow$1 O1$\rightarrow$2 (C1-7-6-2) O1$\rightarrow$16 (C1-16-17-2) O7$\rightarrow$8 O3$\rightarrow$1 (C1-7-8-3) S:13-3-1-16

\noindent
{\bf 49.} MC51 10$\rightarrow$7 O6$\rightarrow$7 (C6-10-7) O9$\rightarrow$7 (C7-10-9) O8$\rightarrow$7 (C7-10-9-8) O1$\rightarrow$7 O3$\rightarrow$1 (C1-7-8-3) O1$\rightarrow$2 (C1-7-6-2) O1$\rightarrow$16 (C1-16-17-2) S:13-3-1-16

\noindent
{\bf 50.} MC49 10$\rightarrow$6 O10$\rightarrow$11 (C5-11-10-6) O10$\rightarrow$7 O7$\rightarrow$6 (C6-11-10-7) O7$\rightarrow$1 O1$\rightarrow$2 (C1-7-6-2) O1$\rightarrow$16 (C1-16-17-2) O1$\rightarrow$3 (C1-16-13-3) O7$\rightarrow$8 O8$\rightarrow$3 (C1-7-8-3) O10$\rightarrow$9 O9$\rightarrow$8 (C7-10-9-8) O7$\rightarrow$9 (C6-10-9-7) O11$\rightarrow$9 (C6-11-9-7) S:11-9-8-3-5

\noindent
{\bf 51.} MC2 16$\rightarrow$14 O12$\rightarrow$14 (C12-14-16-13) O12$\rightarrow$4 (C4-13-14-12) B5$\rightarrow$6 (Copy 52) O5$\rightarrow$3 (C3-13-6-5) O5$\rightarrow$4 (C3-13-4-5) B2$\rightarrow$6 (Copy 53) O2$\rightarrow$3 (C2-6-5-3) O2$\rightarrow$4 (C2-4-5-3) B2$\rightarrow$17 (Copy 54) B16$\rightarrow$17 (Copy 55) B15$\rightarrow$17 (Copy 56) O15$\rightarrow$14 (C14-16-17-15) O15$\rightarrow$12 (C12-15-14-13) O15$\rightarrow$16 (C12-15-16-13) B6$\rightarrow$11 (Copy 57) O5$\rightarrow$11 (C5-11-6) O12$\rightarrow$11 (C4-12-11-5) B6$\rightarrow$10 (Copy 58) O11$\rightarrow$10 (C5-11-10-6) O7$\rightarrow$10 O6$\rightarrow$7 (C6-11-10-7) O1$\rightarrow$7 O2$\rightarrow$1 (C1-7-6-2) O16$\rightarrow$1 (C1-16-17-2) O3$\rightarrow$1 (C1-16-13-3) O8$\rightarrow$7 O3$\rightarrow$8 (C1-7-8-3) O9$\rightarrow$10 O8$\rightarrow$9 (C7-10-9-8) O9$\rightarrow$7 (C6-10-9-7) O9$\rightarrow$11 (C6-11-9-7) S:5-3-8-9-11

\noindent
{\bf 52.} MC58 10$\rightarrow$6 O10$\rightarrow$11 (C6-11-10) O9$\rightarrow$11 O10$\rightarrow$9 (C6-11-9-10) B8$\rightarrow$9 (Copy 59) O8$\rightarrow$3 (C3-8-9-11-5) B7$\rightarrow$10 (Copy 60) O7$\rightarrow$6 (C6-10-7) O7$\rightarrow$9 (C7-10-9) O7$\rightarrow$8 (C7-10-9-8) O7$\rightarrow$1 O1$\rightarrow$3 (C1-7-8-3) O2$\rightarrow$1 (C1-7-6-2) O16$\rightarrow$1 (C1-16-17-2) S:13-16-1-3

\noindent
{\bf 53.} MC60 10$\rightarrow$7 O6$\rightarrow$7 (C6-11-10-7) O1$\rightarrow$7 O2$\rightarrow$1 (C1-7-6-2) O16$\rightarrow$1 (C1-16-17-2) O8$\rightarrow$7 O1$\rightarrow$3 (C1-7-8-3) S:13-16-1-3

\noindent
{\bf 54.} MC59 9$\rightarrow$8 O10$\rightarrow$7 O7$\rightarrow$8 (C7-10-9-8) O6$\rightarrow$7 (C6-11-10-7) O1$\rightarrow$7 O2$\rightarrow$1 (C1-7-6-2) O16$\rightarrow$1 (C1-16-17-2) O3$\rightarrow$8 O1$\rightarrow$3 (C1-7-8-3) S:13-16-1-3

\noindent
{\bf 55.} MC57 11$\rightarrow$6 O11$\rightarrow$12 (C6-13-12-11) O11$\rightarrow$5 (C4-12-11-5) O11$\rightarrow$10 O10$\rightarrow$6 (C5-11-10-6) B6$\rightarrow$7 (Copy 61) O10$\rightarrow$7 (C6-10-7) O1$\rightarrow$7 O2$\rightarrow$1 (C1-7-6-2) O16$\rightarrow$1 (C1-16-17-2) O3$\rightarrow$1 (C1-16-13-3) O8$\rightarrow$7 O3$\rightarrow$8 (C1-7-8-3) O10$\rightarrow$9 O9$\rightarrow$7 (C6-10-9-7) O11$\rightarrow$9 (C9-11-10) O8$\rightarrow$9 (C7-10-9-8) S:11-5-3-8-9

\noindent
{\bf 56.} MC61 7$\rightarrow$6 O7$\rightarrow$10 (C6-11-10-7) B7$\rightarrow$8 (Copy 62) B3$\rightarrow$8 (Copy 63) O9$\rightarrow$8 O11$\rightarrow$9 (C3-8-9-11-5) O7$\rightarrow$9 (C6-11-9-7) O10$\rightarrow$9 (C6-10-9-7) S:7-10-9-8

\noindent
{\bf 57.} MC63 8$\rightarrow$3 O7$\rightarrow$1 O1$\rightarrow$3 (C1-7-8-3) O2$\rightarrow$1 (C1-7-6-2) O16$\rightarrow$1 (C1-16-17-2) S:13-16-1-3

\noindent
{\bf 58.} MC62 8$\rightarrow$7 O9$\rightarrow$10 O8$\rightarrow$9 (C7-10-9-8) O9$\rightarrow$7 (C6-10-9-7) O9$\rightarrow$11 (C6-11-9-7) O8$\rightarrow$3 (C2-6-7-8-3) S:8-9-11-5-3

\noindent
{\bf 59.} MC56 17$\rightarrow$15 O16$\rightarrow$15 (C15-17-16) O12$\rightarrow$15 (C12-15-16-13) O14$\rightarrow$15 (C12-15-14-13) B6$\rightarrow$11 (Copy 64) O5$\rightarrow$11 (C5-11-6) O12$\rightarrow$11 (C4-12-11-5) B6$\rightarrow$10 (Copy 65) O11$\rightarrow$10 (C5-11-10-6) O7$\rightarrow$10 O6$\rightarrow$7 (C6-11-10-7) O1$\rightarrow$7 O2$\rightarrow$1 (C1-7-6-2) O16$\rightarrow$1 (C1-16-17-2) O3$\rightarrow$1 (C1-16-13-3) O8$\rightarrow$7 O3$\rightarrow$8 (C1-7-8-3) O9$\rightarrow$10 O8$\rightarrow$9 (C7-10-9-8) O9$\rightarrow$7 (C6-10-9-7) O9$\rightarrow$11 (C6-11-9-7) S:5-3-8-9-11

\noindent
{\bf 60.} MC65 10$\rightarrow$6 O10$\rightarrow$11 (C6-11-10) O9$\rightarrow$11 O10$\rightarrow$9 (C6-11-9-10) B8$\rightarrow$9 (Copy 66) O8$\rightarrow$3 (C3-8-9-11-5) B7$\rightarrow$10 (Copy 67) O7$\rightarrow$6 (C6-10-7) O7$\rightarrow$9 (C7-10-9) O7$\rightarrow$8 (C7-10-9-8) O7$\rightarrow$1 O1$\rightarrow$3 (C1-7-8-3) O2$\rightarrow$1 (C1-7-6-2) O16$\rightarrow$1 (C1-16-17-2) S:13-16-1-3

\noindent
{\bf 61.} MC67 10$\rightarrow$7 O6$\rightarrow$7 (C6-11-10-7) O1$\rightarrow$7 O2$\rightarrow$1 (C1-7-6-2) O16$\rightarrow$1 (C1-16-17-2) O8$\rightarrow$7 O1$\rightarrow$3 (C1-7-8-3) S:13-16-1-3

\noindent
{\bf 62.} MC66 9$\rightarrow$8 O10$\rightarrow$7 O7$\rightarrow$8 (C7-10-9-8) O6$\rightarrow$7 (C6-11-10-7) O1$\rightarrow$7 O2$\rightarrow$1 (C1-7-6-2) O16$\rightarrow$1 (C1-16-17-2) O3$\rightarrow$8 O1$\rightarrow$3 (C1-7-8-3) S:13-16-1-3

\noindent
{\bf 63.} MC64 11$\rightarrow$6 O11$\rightarrow$12 (C6-13-12-11) O11$\rightarrow$5 (C4-12-11-5) O11$\rightarrow$10 O10$\rightarrow$6 (C5-11-10-6) B6$\rightarrow$7 (Copy 68) O10$\rightarrow$7 (C6-10-7) O1$\rightarrow$7 O2$\rightarrow$1 (C1-7-6-2) O16$\rightarrow$1 (C1-16-17-2) O3$\rightarrow$1 (C1-16-13-3) O8$\rightarrow$7 O3$\rightarrow$8 (C1-7-8-3) O10$\rightarrow$9 O9$\rightarrow$7 (C6-10-9-7) O11$\rightarrow$9 (C9-11-10) O8$\rightarrow$9 (C7-10-9-8) S:11-5-3-8-9

\noindent
{\bf 64.} MC68 7$\rightarrow$6 O7$\rightarrow$10 (C6-11-10-7) B7$\rightarrow$8 (Copy 69) B3$\rightarrow$8 (Copy 70) O9$\rightarrow$8 O11$\rightarrow$9 (C3-8-9-11-5) O7$\rightarrow$9 (C6-11-9-7) O10$\rightarrow$9 (C6-10-9-7) S:7-10-9-8

\noindent
{\bf 65.} MC70 8$\rightarrow$3 O7$\rightarrow$1 O1$\rightarrow$3 (C1-7-8-3) O2$\rightarrow$1 (C1-7-6-2) O16$\rightarrow$1 (C1-16-17-2) S:13-16-1-3

\noindent
{\bf 66.} MC69 8$\rightarrow$7 O9$\rightarrow$10 O8$\rightarrow$9 (C7-10-9-8) O9$\rightarrow$7 (C6-10-9-7) O9$\rightarrow$11 (C6-11-9-7) O8$\rightarrow$3 (C2-6-7-8-3) S:8-9-11-5-3

\noindent
{\bf 67.} MC55 17$\rightarrow$16 O1$\rightarrow$16 O2$\rightarrow$1 (C1-16-17-2) O1$\rightarrow$3 (C1-16-13-3) O17$\rightarrow$15 O15$\rightarrow$14 (C14-16-17-15) O15$\rightarrow$12 (C12-15-14-13) S:2-17-15-12-4

\noindent
{\bf 68.} MC54 17$\rightarrow$2 O17$\rightarrow$16 (C2-17-16-13-3) O17$\rightarrow$15 O15$\rightarrow$14 (C14-16-17-15) O15$\rightarrow$12 (C12-15-14-13) O15$\rightarrow$16 (C12-15-16-13) B6$\rightarrow$11 (Copy 71) O5$\rightarrow$11 (C5-11-6) O12$\rightarrow$11 (C4-12-11-5) B6$\rightarrow$10 (Copy 72) O11$\rightarrow$10 (C5-11-10-6) O7$\rightarrow$10 O6$\rightarrow$7 (C6-11-10-7) O1$\rightarrow$7 O2$\rightarrow$1 (C1-7-6-2) O16$\rightarrow$1 (C1-16-17-2) O3$\rightarrow$1 (C1-16-13-3) O8$\rightarrow$7 O3$\rightarrow$8 (C1-7-8-3) O9$\rightarrow$10 O8$\rightarrow$9 (C7-10-9-8) O9$\rightarrow$7 (C6-10-9-7) O9$\rightarrow$11 (C6-11-9-7) S:5-3-8-9-11

\noindent
{\bf 69.} MC72 10$\rightarrow$6 O10$\rightarrow$11 (C6-11-10) O9$\rightarrow$11 O10$\rightarrow$9 (C6-11-9-10) B8$\rightarrow$9 (Copy 73) O8$\rightarrow$3 (C3-8-9-11-5) B7$\rightarrow$10 (Copy 74) O7$\rightarrow$6 (C6-10-7) O7$\rightarrow$9 (C7-10-9) O7$\rightarrow$8 (C7-10-9-8) O7$\rightarrow$1 O1$\rightarrow$3 (C1-7-8-3) O2$\rightarrow$1 (C1-7-6-2) O16$\rightarrow$1 (C1-16-17-2) S:13-16-1-3

\noindent
{\bf 70.} MC74 10$\rightarrow$7 O6$\rightarrow$7 (C6-11-10-7) O1$\rightarrow$7 O2$\rightarrow$1 (C1-7-6-2) O16$\rightarrow$1 (C1-16-17-2) O8$\rightarrow$7 O1$\rightarrow$3 (C1-7-8-3) S:13-16-1-3

\noindent
{\bf 71.} MC73 9$\rightarrow$8 O10$\rightarrow$7 O7$\rightarrow$8 (C7-10-9-8) O6$\rightarrow$7 (C6-11-10-7) O1$\rightarrow$7 O2$\rightarrow$1 (C1-7-6-2) O16$\rightarrow$1 (C1-16-17-2) O3$\rightarrow$8 O1$\rightarrow$3 (C1-7-8-3) S:13-16-1-3

\noindent
{\bf 72.} MC71 11$\rightarrow$6 O11$\rightarrow$12 (C6-13-12-11) O11$\rightarrow$5 (C4-12-11-5) O11$\rightarrow$10 O10$\rightarrow$6 (C5-11-10-6) B6$\rightarrow$7 (Copy 75) O10$\rightarrow$7 (C6-10-7) O1$\rightarrow$7 O2$\rightarrow$1 (C1-7-6-2) O16$\rightarrow$1 (C1-16-17-2) O3$\rightarrow$1 (C1-16-13-3) O8$\rightarrow$7 O3$\rightarrow$8 (C1-7-8-3) O10$\rightarrow$9 O9$\rightarrow$7 (C6-10-9-7) O11$\rightarrow$9 (C9-11-10) O8$\rightarrow$9 (C7-10-9-8) S:11-5-3-8-9

\noindent
{\bf 73.} MC75 7$\rightarrow$6 O7$\rightarrow$10 (C6-11-10-7) B7$\rightarrow$8 (Copy 76) B3$\rightarrow$8 (Copy 77) O9$\rightarrow$8 O11$\rightarrow$9 (C3-8-9-11-5) O7$\rightarrow$9 (C6-11-9-7) O10$\rightarrow$9 (C6-10-9-7) S:7-10-9-8

\noindent
{\bf 74.} MC77 8$\rightarrow$3 O7$\rightarrow$1 O1$\rightarrow$3 (C1-7-8-3) O2$\rightarrow$1 (C1-7-6-2) O16$\rightarrow$1 (C1-16-17-2) S:13-16-1-3

\noindent
{\bf 75.} MC76 8$\rightarrow$7 O9$\rightarrow$10 O8$\rightarrow$9 (C7-10-9-8) O9$\rightarrow$7 (C6-10-9-7) O9$\rightarrow$11 (C6-11-9-7) O8$\rightarrow$3 (C2-6-7-8-3) S:8-9-11-5-3

\noindent
{\bf 76.} MC53 6$\rightarrow$2 O3$\rightarrow$2 (C2-6-5-3) O4$\rightarrow$2 (C2-4-5-3) B2$\rightarrow$17 (Copy 78) O16$\rightarrow$17 (C2-17-16-13-3) O15$\rightarrow$17 O12$\rightarrow$15 (C2-17-15-12-4) O14$\rightarrow$15 (C12-15-14-13) S:16-14-15-17

\noindent
{\bf 77.} MC78 17$\rightarrow$2 B16$\rightarrow$17 (Copy 79) O16$\rightarrow$1 O1$\rightarrow$2 (C1-16-17-2) O3$\rightarrow$1 (C1-16-13-3) B6$\rightarrow$11 (Copy 80) O5$\rightarrow$11 (C5-11-6) O12$\rightarrow$11 (C4-12-11-5) B1$\rightarrow$7 (Copy 81) O6$\rightarrow$7 (C1-7-6-2) O8$\rightarrow$7 O3$\rightarrow$8 (C1-7-8-3) B7$\rightarrow$10 (Copy 82) O6$\rightarrow$10 (C6-10-7) O11$\rightarrow$10 (C5-11-10-6) O9$\rightarrow$10 O8$\rightarrow$9 (C7-10-9-8) O9$\rightarrow$7 (C6-10-9-7) O9$\rightarrow$11 (C6-11-9-7) S:5-3-8-9-11

\noindent
{\bf 78.} MC82 10$\rightarrow$7 O10$\rightarrow$11 (C6-11-10-7) O10$\rightarrow$6 (C5-11-10-6) O10$\rightarrow$9 O9$\rightarrow$7 (C6-10-9-7) O9$\rightarrow$11 (C6-11-9-7) O8$\rightarrow$9 (C7-10-9-8) S:5-3-8-9-11

\noindent
{\bf 79.} MC81 7$\rightarrow$1 O7$\rightarrow$6 (C1-7-6-2) O9$\rightarrow$11 O7$\rightarrow$9 (C6-11-9-7) O10$\rightarrow$11 O7$\rightarrow$10 (C6-11-10-7) O10$\rightarrow$6 (C5-11-10-6) O10$\rightarrow$9 (C6-10-9-7) O8$\rightarrow$9 O7$\rightarrow$8 (C7-10-9-8) O3$\rightarrow$8 (C1-7-8-3) S:5-3-8-9-11

\noindent
{\bf 80.} MC80 11$\rightarrow$6 O11$\rightarrow$12 (C6-13-12-11) O11$\rightarrow$5 (C4-12-11-5) O11$\rightarrow$10 O10$\rightarrow$6 (C5-11-10-6) B9$\rightarrow$10 (Copy 83) O9$\rightarrow$7 O7$\rightarrow$6 (C6-10-9-7) O7$\rightarrow$1 (C1-7-6-2) O9$\rightarrow$11 (C6-11-9-7) O7$\rightarrow$10 (C6-11-10-7) O8$\rightarrow$3 O9$\rightarrow$8 (C3-8-9-11-5) O8$\rightarrow$7 (C1-7-8-3) S:9-8-7-10

\noindent
{\bf 81.} MC83 10$\rightarrow$9 O11$\rightarrow$9 (C9-11-10) B1$\rightarrow$7 (Copy 84) O6$\rightarrow$7 (C1-7-6-2) O10$\rightarrow$7 (C6-10-7) O8$\rightarrow$7 O3$\rightarrow$8 (C1-7-8-3) O9$\rightarrow$7 (C6-10-9-7) O8$\rightarrow$9 (C7-10-9-8) S:11-5-3-8-9

\noindent
{\bf 82.} MC84 7$\rightarrow$1 O7$\rightarrow$6 (C1-7-6-2) O7$\rightarrow$9 (C6-10-9-7) O7$\rightarrow$10 (C6-11-10-7) O8$\rightarrow$9 O7$\rightarrow$8 (C7-10-9-8) O3$\rightarrow$8 (C1-7-8-3) S:11-5-3-8-9

\noindent
{\bf 83.} MC79 17$\rightarrow$16 O17$\rightarrow$15 O15$\rightarrow$14 (C14-16-17-15) O15$\rightarrow$12 (C12-15-14-13) S:17-15-12-4-2

\noindent
{\bf 84.} MC52 6$\rightarrow$5 O3$\rightarrow$5 (C3-13-6-5) O4$\rightarrow$5 (C3-13-4-5) O12$\rightarrow$11 O11$\rightarrow$5 (C4-12-11-5) O6$\rightarrow$11 (C6-13-12-11) B2$\rightarrow$6 (Copy 85) O2$\rightarrow$3 (C2-6-5-3) O2$\rightarrow$4 (C2-4-5-3) B6$\rightarrow$10 (Copy 86) O11$\rightarrow$10 (C5-11-10-6) O7$\rightarrow$10 O6$\rightarrow$7 (C6-11-10-7) O1$\rightarrow$7 O2$\rightarrow$1 (C1-7-6-2) B1$\rightarrow$16 (Copy 87) O17$\rightarrow$16 O2$\rightarrow$17 (C1-16-17-2) O1$\rightarrow$3 (C1-16-13-3) O17$\rightarrow$15 O15$\rightarrow$14 (C14-16-17-15) O15$\rightarrow$12 (C12-15-14-13) S:2-17-15-12-4

\noindent
{\bf 85.} MC87 16$\rightarrow$1 O3$\rightarrow$1 (C1-16-13-3) O8$\rightarrow$7 O3$\rightarrow$8 (C1-7-8-3) O9$\rightarrow$10 O8$\rightarrow$9 (C7-10-9-8) O9$\rightarrow$7 (C6-10-9-7) O9$\rightarrow$11 (C6-11-9-7) S:3-8-9-11-5

\noindent
{\bf 86.} MC86 10$\rightarrow$6 O10$\rightarrow$11 (C6-11-10) O9$\rightarrow$11 O10$\rightarrow$9 (C6-11-9-10) B8$\rightarrow$9 (Copy 88) O8$\rightarrow$3 (C3-8-9-11-5) B15$\rightarrow$16 (Copy 89) O15$\rightarrow$14 (C14-16-15) O15$\rightarrow$12 (C12-15-14-13) B2$\rightarrow$17 (Copy 90) O15$\rightarrow$17 (C2-17-15-12-4) O16$\rightarrow$17 (C14-16-17-15) B7$\rightarrow$10 (Copy 91) O7$\rightarrow$6 (C6-10-7) O7$\rightarrow$9 (C7-10-9) O7$\rightarrow$8 (C7-10-9-8) O7$\rightarrow$1 O1$\rightarrow$3 (C1-7-8-3) O2$\rightarrow$1 (C1-7-6-2) O16$\rightarrow$1 (C1-16-17-2) S:13-16-1-3

\noindent
{\bf 87.} MC91 10$\rightarrow$7 O6$\rightarrow$7 (C6-11-10-7) O1$\rightarrow$7 O2$\rightarrow$1 (C1-7-6-2) O16$\rightarrow$1 (C1-16-17-2) O8$\rightarrow$7 O1$\rightarrow$3 (C1-7-8-3) S:13-16-1-3

\noindent
{\bf 88.} MC90 17$\rightarrow$2 O17$\rightarrow$16 (C2-17-16-13-3) O17$\rightarrow$15 (C14-16-17-15) B7$\rightarrow$10 (Copy 92) O7$\rightarrow$6 (C6-10-7) O7$\rightarrow$9 (C7-10-9) O7$\rightarrow$8 (C7-10-9-8) O7$\rightarrow$1 O1$\rightarrow$3 (C1-7-8-3) O2$\rightarrow$1 (C1-7-6-2) O16$\rightarrow$1 (C1-16-17-2) S:13-16-1-3

\noindent
{\bf 89.} MC92 10$\rightarrow$7 O6$\rightarrow$7 (C6-11-10-7) O1$\rightarrow$7 O2$\rightarrow$1 (C1-7-6-2) O16$\rightarrow$1 (C1-16-17-2) O8$\rightarrow$7 O1$\rightarrow$3 (C1-7-8-3) S:13-16-1-3

\noindent
{\bf 90.} MC89 16$\rightarrow$15 O12$\rightarrow$15 (C12-15-16-13) O14$\rightarrow$15 (C12-15-14-13) O16$\rightarrow$17 O17$\rightarrow$15 (C14-16-17-15) O2$\rightarrow$17 (C2-17-16-13-3) B7$\rightarrow$10 (Copy 93) O7$\rightarrow$6 (C6-10-7) O7$\rightarrow$9 (C7-10-9) O7$\rightarrow$8 (C7-10-9-8) O7$\rightarrow$1 O1$\rightarrow$3 (C1-7-8-3) O2$\rightarrow$1 (C1-7-6-2) O16$\rightarrow$1 (C1-16-17-2) S:13-16-1-3

\noindent
{\bf 91.} MC93 10$\rightarrow$7 O6$\rightarrow$7 (C6-11-10-7) O1$\rightarrow$7 O2$\rightarrow$1 (C1-7-6-2) O16$\rightarrow$1 (C1-16-17-2) O8$\rightarrow$7 O1$\rightarrow$3 (C1-7-8-3) S:13-16-1-3

\noindent
{\bf 92.} MC88 9$\rightarrow$8 O10$\rightarrow$7 O7$\rightarrow$8 (C7-10-9-8) O6$\rightarrow$7 (C6-11-10-7) O1$\rightarrow$7 O2$\rightarrow$1 (C1-7-6-2) O3$\rightarrow$8 O1$\rightarrow$3 (C1-7-8-3) O1$\rightarrow$16 (C1-16-13-3) O17$\rightarrow$16 O2$\rightarrow$17 (C1-16-17-2) O9$\rightarrow$7 (C6-10-9-7) O17$\rightarrow$15 O15$\rightarrow$14 (C14-16-17-15) O15$\rightarrow$12 (C12-15-14-13) S:2-17-15-12-4

\noindent
{\bf 93.} MC85 6$\rightarrow$2 O3$\rightarrow$2 (C2-6-5-3) O4$\rightarrow$2 (C2-4-5-3) B6$\rightarrow$10 (Copy 94) O11$\rightarrow$10 (C5-11-10-6) O7$\rightarrow$10 O6$\rightarrow$7 (C6-11-10-7) B7$\rightarrow$8 (Copy 95) O3$\rightarrow$8 (C2-6-7-8-3) B15$\rightarrow$16 (Copy 96) O15$\rightarrow$14 (C14-16-15) O15$\rightarrow$12 (C12-15-14-13) O17$\rightarrow$2 O15$\rightarrow$17 (C2-17-15-12-4) O16$\rightarrow$17 (C14-16-17-15) O16$\rightarrow$1 O1$\rightarrow$2 (C1-16-17-2) O1$\rightarrow$7 (C1-7-6-2) O1$\rightarrow$3 (C1-7-8-3) S:13-16-1-3

\noindent
{\bf 94.} MC96 16$\rightarrow$15 O12$\rightarrow$15 (C12-15-16-13) O14$\rightarrow$15 (C12-15-14-13) O16$\rightarrow$17 O17$\rightarrow$15 (C14-16-17-15) O17$\rightarrow$2 (C2-17-15-12-4) O16$\rightarrow$1 O1$\rightarrow$2 (C1-16-17-2) O1$\rightarrow$7 (C1-7-6-2) O1$\rightarrow$3 (C1-7-8-3) S:13-16-1-3

\noindent
{\bf 95.} MC95 8$\rightarrow$7 O9$\rightarrow$10 O8$\rightarrow$9 (C7-10-9-8) O9$\rightarrow$7 (C6-10-9-7) O9$\rightarrow$11 (C6-11-9-7) O8$\rightarrow$3 (C3-8-9-11-5) B12$\rightarrow$15 (Copy 97) O14$\rightarrow$15 (C12-15-14-13) O16$\rightarrow$15 (C14-16-15) O16$\rightarrow$17 O17$\rightarrow$15 (C14-16-17-15) O17$\rightarrow$2 (C2-17-15-12-4) O16$\rightarrow$1 O1$\rightarrow$2 (C1-16-17-2) O1$\rightarrow$7 (C1-7-6-2) O1$\rightarrow$3 (C1-7-8-3) S:13-16-1-3

\noindent
{\bf 96.} MC97 15$\rightarrow$12 O15$\rightarrow$14 (C12-15-14) O15$\rightarrow$16 (C12-15-16-13) O17$\rightarrow$2 O15$\rightarrow$17 (C2-17-15-12-4) O16$\rightarrow$17 (C14-16-17-15) O16$\rightarrow$1 O1$\rightarrow$2 (C1-16-17-2) O1$\rightarrow$7 (C1-7-6-2) O1$\rightarrow$3 (C1-7-8-3) S:13-16-1-3

\noindent
{\bf 97.} MC94 10$\rightarrow$6 O10$\rightarrow$11 (C6-11-10) O9$\rightarrow$11 O10$\rightarrow$9 (C6-11-9-10) B8$\rightarrow$9 (Copy 98) O8$\rightarrow$3 (C3-8-9-11-5) B15$\rightarrow$16 (Copy 99) O15$\rightarrow$14 (C14-16-15) O15$\rightarrow$12 (C12-15-14-13) O17$\rightarrow$2 O15$\rightarrow$17 (C2-17-15-12-4) O16$\rightarrow$17 (C14-16-17-15) O16$\rightarrow$1 O1$\rightarrow$2 (C1-16-17-2) O3$\rightarrow$1 (C1-16-13-3) O7$\rightarrow$1 O8$\rightarrow$7 (C1-7-8-3) O7$\rightarrow$6 (C1-7-6-2) O7$\rightarrow$9 (C6-10-9-7) O7$\rightarrow$10 (C6-11-10-7) S:8-7-10-9

\noindent
{\bf 98.} MC99 16$\rightarrow$15 O12$\rightarrow$15 (C12-15-16-13) O14$\rightarrow$15 (C12-15-14-13) O16$\rightarrow$17 O17$\rightarrow$15 (C14-16-17-15) O17$\rightarrow$2 (C2-17-15-12-4) O16$\rightarrow$1 O1$\rightarrow$2 (C1-16-17-2) O3$\rightarrow$1 (C1-16-13-3) O7$\rightarrow$1 O8$\rightarrow$7 (C1-7-8-3) O7$\rightarrow$6 (C1-7-6-2) O7$\rightarrow$9 (C6-10-9-7) O7$\rightarrow$10 (C6-11-10-7) S:8-7-10-9

\noindent
{\bf 99.} MC98 9$\rightarrow$8 O10$\rightarrow$7 O7$\rightarrow$8 (C7-10-9-8) O6$\rightarrow$7 (C6-11-10-7) O9$\rightarrow$7 (C6-10-9-7) O3$\rightarrow$8 (C2-6-7-8-3) B12$\rightarrow$15 (Copy 100) O14$\rightarrow$15 (C12-15-14-13) O16$\rightarrow$15 (C14-16-15) O16$\rightarrow$17 O17$\rightarrow$15 (C14-16-17-15) O17$\rightarrow$2 (C2-17-15-12-4) O16$\rightarrow$1 O1$\rightarrow$2 (C1-16-17-2) O1$\rightarrow$7 (C1-7-6-2) O1$\rightarrow$3 (C1-7-8-3) S:13-16-1-3

\noindent
{\bf 100.} MC100 15$\rightarrow$12 O15$\rightarrow$14 (C12-15-14) O15$\rightarrow$16 (C12-15-16-13) O17$\rightarrow$2 O15$\rightarrow$17 (C2-17-15-12-4) O16$\rightarrow$17 (C14-16-17-15) O16$\rightarrow$1 O1$\rightarrow$2 (C1-16-17-2) O1$\rightarrow$7 (C1-7-6-2) O1$\rightarrow$3 (C1-7-8-3) S:13-16-1-3

\end{tiny}

\end{document}